\documentclass[12pt]{amsart}
\usepackage{amsmath}
\usepackage{amstext}
\usepackage{amssymb}
\usepackage{amsthm}
\swapnumbers
\theoremstyle{plain}
\newtheorem{thm}{Theorem}[section]
\newtheorem{embthm}[thm]{Embedding Theorem}
\newtheorem{lem}[thm]{Lemma}

\newtheorem{prop}[thm]{Proposition}
\newtheorem{cor}[thm]{Corollary}

\theoremstyle{definition}
\newtheorem{defi}[thm]{Definition}
\newtheorem{defs}[thm]{Definitions}
\newtheorem{bascorr}[thm]{The Basic Correspondence}

\newtheorem{ex}[thm]{Example}
\newtheorem{exs}[thm]{Examples}

\newtheorem{ntn}[thm]{Notation}
\newtheorem{rmd}[thm]{Reminder}
\newtheorem{rmds}[thm]{Reminders}
\theoremstyle{remark}
\newtheorem*{note}{Note}
\newtheorem{rmk}[thm]{Remark}
\newtheorem{rmks}[thm]{Remarks}

 \DeclareMathOperator{\ext}{exten}

 \DeclareMathOperator{\ass}{ass}  \DeclareMathOperator{\Max}{Max}

 \DeclareMathOperator{\Spec}{Spec}
 
 \DeclareMathOperator{\ann}{ann}
\DeclareMathOperator{\grann}{gr-ann} 
\def\Z{\mathbb Z}
\def\N{\mathbb N}

\def\fa{{\mathfrak{a}}}
\def\fA{{\mathfrak{A}}}
\def\fb{{\mathfrak{b}}}

\def\fc{{\mathfrak{c}}}

\def\fd{{\mathfrak{d}}}

\def\fg{{\mathfrak{g}}}
\def\fh{{\mathfrak{h}}}

\def\fm{{\mathfrak{m}}}

\def\fp{{\mathfrak{p}}}

\def\fq{{\mathfrak{q}}}

\def\nn{\relax\ifmmode{\mathbb N_{0}}\else$\mathbb N_{0}$\fi}
\def\lra{\longrightarrow}

\begin{document}



\title[$S$-tight closure
over an $F$-pure local ring]{Tight closure with respect to a
multiplicatively closed subset of an $F$-pure local ring}
\author{RODNEY Y. SHARP}
\address{{\it School of Mathematics and Statistics\\
University of Sheffield\\ Hicks Building\\ Sheffield S3 7RH\\ United
Kingdom}} \email{R.Y.Sharp@sheffield.ac.uk}


\date{\today}


\begin{abstract}
Let $R$ be a (commutative Noetherian) local ring of prime
characteristic that is $F$-pure. This paper studies a certain finite
set ${\mathcal I}$ of radical ideals of $R$ that is naturally
defined by the injective envelope $E$ of the simple $R$-module. This
set ${\mathcal I}$ contains $0$ and $R$, and is closed under taking
primary components. For a multiplicatively closed subset $S$ of $R$,
the concept of {\em tight closure with respect to $S$\/}, or {\em
$S$-tight closure\/}, is discussed, together with associated
concepts of $S$-test element and $S$-test ideal. It is shown that an
ideal $\fa$ of $R$ belongs to ${\mathcal I}$ if and only if it is
the $S'$-test ideal of $R$ for some multiplicatively closed subset
$S'$ of $R$. When $R$ is complete, ${\mathcal I}$ is also `closed
under taking test ideals', in the following sense: for each proper
ideal $\fc$ in ${\mathcal I}$, it turns out that $R/\fc$ is again
$F$-pure, and if $\fg$ and $\fh$ are the unique ideals of $R$ that
contain $\fc$ and are such that $\fg/\fc$ is the (tight closure)
test ideal of $R/\fc$ and $\fh/\fc$ is the big test ideal of
$R/\fc$, then both $\fg$ and $\fh$ belong to ${\mathcal I}$. The
paper ends with several examples.
\end{abstract}

\maketitle

\setcounter{section}{-1}
\section{\it Introduction}
\label{intro}

This paper is concerned with a (commutative Noetherian) local ring
$R$ having maximal ideal $\fm$ and prime characteristic $p$; the
Frobenius homomorphism $f : R \lra R$, for which $f(r) = r^p$ for
all $r \in R$, will play a central r\^ole. Let us (temporarily) use
$R_f$ to denote the Abelian group $R$ considered as an
$(R,R)$-bimodule with $r_1 \cdot r \cdot r_2 = r_1rr_2^p$ for all
$r, r_1, r_2 \in R$.

The ring $R$ is said to be {\em $F$-pure\/} if and only if, for
every $R$-module $M$, the map $\alpha_M : M \lra R_f\otimes_RM$ for
which $\alpha_M(m) = 1\otimes m$, for all $m \in M$, is injective.
This property can be reformulated in terms of certain left modules
over the skew polynomial ring $R[x,f]$ associated to $R$ and $f$ in
the indeterminate $x$ over $R$, which we refer to as the {\em
Frobenius skew polynomial ring over $R$}. Recall that $R[x,f]$ is,
as a left $R$-module, freely generated by $(x^i)_{i \in \nn}$ (we
use $\nn$ (respectively $\N$) to denote the set of non-negative
(respectively positive) integers),
 and so consists
 of all polynomials $\sum_{i = 0}^n r_i x^i$, where  $n \in \nn$
 and  $r_0,\ldots,r_n \in R$; however, its multiplication is subject to the
 rule
 $
  xr = f(r)x = r^px$ for all $r \in R.$
Note that $R[x,f]$ can be considered as a positively-graded ring
$R[x,f] = \bigoplus_{n=0}^{\infty} R[x,f]_n$, with $R[x,f]_n = Rx^n$
for $n \in \nn$. If we endow $Rx^n$ with its natural structure as an
$(R,R)$-bimodule (inherited from its being a graded component of
$R[x,f]$), then $Rx^n$ is isomorphic (as $(R,R)$-bimodule) to $R$
viewed as a left $R$-module in the natural way and as a right
$R$-module via $f^n$, the $n$th iterate of the Frobenius ring
homomorphism. In particular, $Rx\cong R_f$ as $(R,R)$-bimodules.

A left $R[x,f]$-module $G$ is said to be {\em $x$-torsion-free}
precisely when $xg = 0$, where $g \in G$, implies that $g = 0$. We
can say that $R$ is $F$-pure if and only if, for every $R$-module
$M$, the graded left $R[x,f]$-module
$
R[x,f] \otimes_R M = \bigoplus_{n\in\nn}Rx^n\otimes_R M
$
is $x$-torsion-free. We shall also use the following alternative
characterization of $F$-purity.

\begin{thm} [{\cite[Theorem 3.2]{Fpurhastest}}]
\label{fp.2} The local ring $(R,\fm)$ is $F$-pure if and only if the
$R$-module structure on $E_R(R/\fm)$, the injective envelope of the
simple $R$-module, can be extended to an $x$-torsion-free left
$R[x,f]$-module structure.
\end{thm}

In the paper \cite{ga}, some useful properties of $x$-torsion-free
left $R[x,f]$-modules were developed, and it is appropriate to
recall some of them at this point.

The graded two-sided ideals of $R[x,f]$ are just the subsets of the
form ${\textstyle \bigoplus_{n\in\nn}\fa_nx^n},$ where
$(\fa_n)_{n\in\nn}$ is an ascending sequence of ideals of $R$. (Of
course, such a sequence $\fa_0 \subseteq \fa_1 \subseteq \cdots
\subseteq \fa_n \subseteq \cdots $ is eventually stationary.) Let
$H$ be a left $R[x,f]$-module. An $R[x,f]$-submodule of $H$ is said
to be a {\em special annihilator submodule of $H$\/} if it has the
form
$$\ann_H(\fA) := \{ h \in H : \theta h = 0 \text{~for all~} \theta
\in \fA\}$$ for some {\em graded\/} two-sided ideal $\fA$ of
$R[x,f]$.

We shall use $\mathcal{A}(H)$ to denote the set of special
annihilator submodules of $H$.

The {\em graded annihilator $\grann_{R[x,f]}H$ of $H$\/} is defined
to be the largest graded two-sided ideal of $R[x,f]$ that
annihilates $H$. Note that, if $\grann_{R[x,f]}H =
\bigoplus_{n\in\nn} \fa_nx^n$, then $\fa_0 = (0:_RH)$, which we
shall sometimes refer to as the\/ {\em $R$-annihilator of $H$.}

We shall use $\mathcal{I}(H)$ (or $\mathcal{I}_R(H)$ when it is
desirable to specify the ring $R$) to denote the set of
$R$-annihilators of $R[x,f]$-submodules of $H$.

\begin{bascorr}[{\cite[\S1, \S3]{ga}}]\label{np.2} Let $G$ be an $x$-torsion-free left mo\-d\-ule over $R[x,f]$.

\begin{enumerate} \item By \cite[Lemma 1.9]{ga}, the members of $\mathcal{I}(G)$ are all
radical ideals of $R$; they are referred to as the {\em $G$-special
$R$-ideals\/}; in fact, the graded annihilator of an
$R[x,f]$-submodule $L$ of $G$ is equal to $\bigoplus_{n\in\nn}\fa
x^n$, where $\fa = (0:_RL)$.

\item By \cite[Proposition 1.11]{ga}, there is an order-reversing bijection, $\Theta : \mathcal{A}(G)
\lra \mathcal{I}(G)$ given by
$$
\Theta : N \longmapsto \left(\grann_{R[x,f]}N\right)\cap R =
(0:_RN).
$$
The inverse bijection, $\Theta^{-1} : \mathcal{I}(G) \lra
\mathcal{A}(G),$ also order-reversing, is given by
$$
\Theta^{-1} : \fb \longmapsto \ann_G(\fb R[x,f]).
$$

\item By \cite[Theorem 3.6 and Corollary 3.7]{ga}, the set of $G$-special
$R$-ideals is precisely the set of all finite intersections of prime
$G$-special $R$-ideals (provided one includes the empty
intersection, $R$, which corresponds under the bijection of part
(ii) to the zero special annihilator submodule of $G$). In symbols,
$$
\mathcal{I}(G) = \left\{ \fp_1 \cap \cdots \cap \fp_t : t \in \nn
\mbox{~and~} \fp_1, \ldots, \fp_t \in
\mathcal{I}(G)\cap\Spec(R)\right\}.
$$

\item By \cite[Corollary 3.11]{ga}, when $G$ is Artinian as an
$R$-module, the sets $\mathcal{I}(G)$ and $\mathcal{A}(G)$ are both
finite.
\end{enumerate}
\end{bascorr}

Thus, if $(R,\fm)$ is $F$-pure and we endow $E_R(R/\fm)$ with a
structure as $x$-torsion-free left $R[x,f]$-module (this is
possible, by Theorem \ref{fp.2}), then the resulting set
$\mathcal{I}(E_R(R/\fm))$ is finite. In fact, rather more can be
said.

\begin{thm} [{\cite[Corollary 4.11]{Fpurhastest}}]
\label{in.3} Suppose that $(R,\fm)$ is $F$-pure and local. Then the
left $R[x,f]$-module $R[x,f]\otimes_RE_R(R/\fm)$ is
$x$-torsion-free, and, furthermore, the set ${\mathcal
I}(R[x,f]\otimes_RE_R(R/\fm))$ of
$(R[x,f]\otimes_RE_R(R/\fm))$-special $R$-ideals, is finite.

In fact, for any $x$-torsion-free left $R[x,f]$-module structure on
$E_R(R/\fm)$ that extends its $R$-module structure (and such exist,
by Theorem\/ {\rm \ref{fp.2}}), we have $${\mathcal
I}(R[x,f]\otimes_RE_R(R/\fm)) \subseteq {\mathcal I}(E_R(R/\fm)),$$
and the latter set is finite.
\end{thm}

It turns out that, in the situation of Theorem \ref{in.3}, all the
minimal prime ideals of $R$ belong to ${\mathcal
I}(R[x,f]\otimes_RE_R(R/\fm))$, and the smallest ideal of positive
height in ${\mathcal I}(R[x,f]\otimes_RE_R(R/\fm))$ is the big test
ideal of $R$, that is, the ideal generated by all big test elements
for $R$. (A reminder about big test elements is included in the next
section.)

In short, to each $F$-pure local ring $(R,\fm)$ of characteristic
$p$, there is naturally associated a finite set of radical ideals
${\mathcal I}(R[x,f]\otimes_RE_R(R/\fm))$ of $R$, and the smallest
member of positive height in this set has significance for tight
closure theory. The purpose of this paper is to address the
following question: what can be said about the other members of
${\mathcal I}(R[x,f]\otimes_RE_R(R/\fm))$? We shall show that, for
each multiplicatively closed subset $S$ of $R$, one can define
reasonable concepts of $S$-tight closure, $S$-test element and
$S$-test ideal; it turns out that an ideal $\fb$ of $R$ is a member
of ${\mathcal I}(R[x,f]\otimes_RE_R(R/\fm))$ if and only if it is
the $S$-test ideal of $R$ for some choice of multiplicatively closed
subset $S$ of $R$.

We shall also show that, when $R$ is complete, ${\mathcal
I}(R[x,f]\otimes_RE_R(R/\fm))$ is `closed under taking (big) test
ideals', in the following sense: it turns out that, for each proper
ideal $\fc \in {\mathcal I}(R[x,f]\otimes_RE_R(R/\fm))$, the ring
$R/\fc$ is again $F$-pure, and the set ${\mathcal
I}(R[x,f]\otimes_RE_R(R/\fm))$ has among its membership the unique
ideals $\fg$ and $\fh$ of $R$ such that $\fg,\fh \supseteq \fc$ and
$\fg/\fc = \widetilde{\tau}(R/\fc)$, the big test ideal of $R/\fc$,
and $\fh/\fc = \tau(R/\fc)$, the test ideal of $R/\fc$.

I am grateful to Mordechai Katzman for helpful discussions about the
material in this paper.

\section{Internal $S$-tight closure}
\label{st}

\begin{ntn}
\label{st.0} From this point onwards in the paper, $R$ will denote a
commutative Noetherian ring of prime characteristic $p$. We shall
only assume that $R$ is local when this is explicitly stated; then,
the notation `$(R,\fm)$' will denote that $\fm$ is the maximal ideal
of $R$. As in tight closure theory, we use $R^{\circ}$ to denote the
complement in $R$ of the union of the minimal prime ideals of $R$.
The Frobenius homomorphism on $R$ will always be denoted by $f$.

We shall use $\Phi$ (or $\Phi_R$ when it is desirable to specify
which ring is being considered) to denote the functor
$R[x,f]\otimes_R \;{\scriptscriptstyle \bullet}\;$ from the category
of $R$-modules (and all $R$-homomorphisms) to the category of all
$\nn$-graded left $R[x,f]$-modules (and all homogeneous
$R[x,f]$-homomorphisms). For an $R$-module $M$, we shall identify
$\Phi(M)$ with $\bigoplus_{n\in\nn}Rx^n\otimes_RM$, and (sometimes)
identify its $0$th component $R\otimes_RM$ with $M$, in the obvious
ways.

For $n \in \Z$, we shall denote the $n$th component of a $\Z$-graded
module $L$ by $L_n$.

Throughout the paper, $S$ will denote a multiplicatively closed
subset of $R$. (We require that each multiplicatively closed subset
of $R$ contains $1$.)

Also throughout the paper, $H$ will denote a left $R[x,f]$-module
and $G$ will denote an $x$-torsion-free left $R[x,f]$-module.
\end{ntn}

\begin{defs}
\label{st.2} We define the {\em internal $S$-tight closure of zero
in $H$,\/} denoted $\Delta^S(H)$ (or $\Delta^S_{R}(H)$), to be the
$R[x,f]$-submodule of $H$ given by
$$
\Delta^S(H)= \left\{ h \in H : \text{there exists~} s \in S
\text{~with~} sx^nh = 0 \text{~for all~} n \gg 0 \right\}.
$$
Note that if the left $R[x,f]$-module $H$ is $\Z$-graded, then
$\Delta^S(H)$ is a graded submodule.

Let $M$ be an $R$-module, and consider $\Phi(M)$, as in \ref{st.0}.
Now $\Delta^S(\Phi(M))$ is a graded $R[x,f]$-submodule of $\Phi(M)$;
we refer to the $0$th component of $\Delta^S(\Phi(M))$ as the {\em
$S$-tight closure of\/ $0$ in $M$}, or the {\em tight closure with
respect to $S$ of\/ $0$ in $M$}, and denote it by $0^{*,S}_M$ (or by
$0^{*,S}$ when it is clear what $M$ is).

Thus $0^{*,S}_M$ is the set of all elements $m$ of $M$ for which
there exists $s \in S$ such that, for all $n \gg 0$, we have $sx^n
\otimes m = 0$ in $Rx^n \otimes_RM$. In particular,
$0^{*,R^{\circ}}_M$ is the usual tight closure of $0$ in $M$. (See
M. Hochster and C. Huneke \cite[\S8]{HocHun90}.)

Now let $N$ be an $R$-submodule of $M$. The inverse image of
$0^{*,S}_{M/N}$ under the natural epimorphism $M \lra M/N$ is
defined to be the {\em $S$-tight closure of $N$ in $M$}, and is
denoted by $N^{*,S}_M$ or $N^{*,S}$. Thus $N^{*,S}_M$ is the set of
all elements $m$ of $M$ for which there exists $s \in S$ such that,
for all $n \gg 0$, the element $sx^n \otimes m$ of $Rx^n \otimes_RM$
belongs to the image of $Rx^n \otimes_RN$ in $Rx^n \otimes_RM$. Note
that $N^{*,R^{\circ}}_M$ is the usual tight closure of $N$ in $M$.
(See Hochster--Huneke \cite[\S8]{HocHun90}.) I am grateful to the
referee for pointing out that the $S$-tight closure of $N$ in $M$ is
the tight closure of $N$ in $M$ with respect to ${\mathcal C}$ in
the sense of Hochster--Huneke \cite[Definition (10.1)]{HocHun90},
where ${\mathcal C}$ is $\{sR : s \in S\}$, the family of principal
ideals generated by elements of $S$, directed by reverse inclusion
$\supseteq$.

The $S$-tight closure of $\fa$ in $R$ is referred to simply as the
{\em $S$-tight closure of $\fa$} and is denoted by $\fa^{*,S}$. The
fact that there is a homogeneous $R[x,f]$-isomorphism
$$R[x,f]\otimes_R(R/\fa) \cong \bigoplus_{n\in\nn}R/\fa^{[p^n]}$$
(where the right-hand side has the left $R[x,f]$-module structure
for which $x(r + \fa^{[p^n]}) = r^p + \fa^{[p^{n+1}]}$ for all $r
\in R$) enables one to conclude that
$$
\fa^{*,S} = \left\{ r \in R : \mbox{there exists~} s \in S
\mbox{~with~} sr^{p^n} \in \fa^{[p^n]} \mbox{~for all~} n \gg 0
\right\}.
$$
Thus $\fa^{*,R^{\circ}}$ is the usual tight closure $\fa^*$ of
$\fa$. (See Hochster--Huneke \cite[\S3]{HocHun90}.)
\end{defs}

\begin{exs}
\label{st.3} In the situation of Definition \ref{st.2}, let $S$, $T$
be multiplicatively closed subsets of $R$ with $S \subseteq T$. Then
clearly $\Delta^S(H) \subseteq \Delta^T(H)$.

\begin{enumerate}
\item Note that $\{1\}$ is a multiplicatively closed subset of $R$,
and
\begin{align*}
\Delta^{\{1\}}(H) & = \left\{ h \in H : x^nh = 0 \text{~for all~} n
\gg 0 \right\} \\ &= \{h \in H : \text{there exists~} n \in \nn
\text{~with~} x^nh = 0 \},
\end{align*}
the {\em $x$-torsion submodule\/} $\Gamma_x(H)$ of $H$. Thus
$\Gamma_x(H) \subseteq \Delta^S(H)$ and therefore $(0
:_R\Delta^S(H)) \subseteq (0 :_R \Gamma_x(H))$.
\item Let $M$ be an $R$-module and let $h,n \in \nn$. Endow
$Rx^n$ and $Rx^h$ with their natural structures as $(R,R)$-bimodules
(inherited from their being graded components of $R[x,f]$). Then
there is an isomorphism of (left) $R$-modules $\phi :
Rx^{n+h}\otimes_R M \stackrel{\cong}{\lra}
Rx^{n}\otimes_R(Rx^{h}\otimes_R M)$ for which $\phi(rx^{n+h} \otimes
m) = rx^n \otimes (x^h \otimes m)$ for all $r \in R$ and $m \in M$.

One can use isomorphisms like that described in the above paragraph
to see that
$$\Delta^S(R[x,f]\otimes_RM) = 0^{*,S}_M \oplus 0^{*,S}_{Rx
\otimes_RM} \oplus \cdots \oplus 0^{*,S}_{Rx^n \otimes_RM}\oplus
\cdots.
$$
\end{enumerate}
\end{exs}

\begin{ntn}
\label{st.1} We define $\mathcal{M}^S(G)$, for the $x$-torsion-free
left $R[x,f]$-module $G$, to be the set of minimal members (with
respect to inclusion) of the set $$ \left\{ \fp \in \Spec(R) \cap
\mathcal{I}(G) : \fp \cap S \neq \emptyset\right\} $$ of prime
$G$-special $R$-ideals that meet $S$.

When $\mathcal{M}^S(G)$ is finite, we shall set $\fb := \bigcap_{\fp
\in \mathcal{M}^S(G)}\fp$, although we shall write $\fb^{S,G}$ for
$\fb$ when it is desirable to indicate $S$ and $G$; in that case, it
follows from \ref{np.2}(iii) that $\fb$ is the smallest member of
$\mathcal{I}(G)$ that meets $S$ (and, in particular, $\fb$ is
contained in every other member of $\mathcal{I}(G)$ that meets $S$).
(In the special case where $\mathcal{M}^S(G) = \emptyset$, we
interpret $\fb$ as $R$, the intersection of the empty family of
prime ideals of $R$.)
\end{ntn}

\begin{prop}
\label{st.5} Consider the $x$-torsion-free left $R[x,f]$-module $G$.
The set $\mathcal{M}^S(G)$ (see\/ {\rm \ref{st.1}}) is finite if and
only if $(0:_R\Delta^S(G)) \cap S \neq \emptyset$.

When these conditions are satisfied, and $\fb$ denotes the
intersection of the prime ideals in the finite set
$\mathcal{M}^S(G)$, then $$\Delta^S(G) = \ann_G(\fb R[x,f]) \quad
\mbox{~and~} \quad (0:_R\Delta^S(G)) = \fb.$$
\end{prop}

\begin{proof} ($\Leftarrow$) Set $\fc := (0:_R\Delta^S(G))$ and assume
that there exists $s \in \fc \cap S$. Since $G$ is $x$-torsion-free
and $\Delta^S(G)$ is an $R[x,f]$-submodule of $G$, it follows from
\ref{np.2}(i) that $\fc$ is radical and $\grann_{R[x,f]}\Delta^S(G)
= \fc R[x,f]$. Now
$$
\ann_G(\fc R[x,f]) \subseteq \ann_G((sR) R[x,f]) \subseteq
\Delta^S(G) \subseteq \ann_G(\fc R[x,f]),
$$
and so $\ann_{G}(\fc R[x,f]) = \Delta^S(G)$. Note that $\fc \in
\mathcal{I}(G)$.

If $\fc = R$, then $\Delta^S(G) = 0$, so that $\mathcal {M}^S(G)$ is
empty because a $\fp \in \mathcal{I}(G)\cap\Spec(R)$ with $ \fp\cap
S \neq \emptyset$ must satisfy $\ann_G(\fp R[x,f]) \subseteq
\Delta^S(G) = 0$, and this leads to a contradiction to
\ref{np.2}(ii). We therefore suppose that $\fc \neq R$.

Let $\fc = \fp_1 \cap \cdots \cap \fp_t$ be the minimal primary
decomposition of the (radical) ideal $\fc$. By \cite[Theorem 3.6 and
Corollary 3.7]{ga}, the prime ideals $\fp_1, \ldots, \fp_t$ all
belong to $\mathcal{I}(G)$; they all meet $S$. Since $\Spec (R)$
satisfies the descending chain condition, each member of $\{\fp' \in
\Spec (R)\cap\mathcal{I}(G) : \fp' \cap S \neq \emptyset\}$ contains
a member of $\mathcal{M}^S(G)$. In particular, each of $\fp_1,
\ldots, \fp_t$ contains a member of $\mathcal{M}^S(G)$.

Now let $\fp \in \mathcal{M}^S(G)$. Since $\fp$ meets $S$, we must
have that $\ann_{G}(\fp R[x,f]) \subseteq \Delta^S(G) = \ann_{G}(\fc
R[x,f])$. It now follows from the inclusion-reversing bijective
correspondence of \ref{np.2}(ii) that $\fp \supseteq \fc$, so that
$\fp$ contains one of $\fp_1, \ldots, \fp_t$. We can therefore
conclude that $\fp_1, \ldots, \fp_t$ are precisely the minimal
members of $$\{\fp \in \mathcal{I}(G)\cap\Spec (R) : \fp \cap S \neq
\emptyset\}.$$ Therefore $\mathcal{M}^S(G)$ is finite.

($\Rightarrow$) Assume that $\mathcal{M}^S(G)$ is finite. Then $\fb$
is the smallest member of ${\mathcal I}(G)$ that meets $S$.
Consequently, $\ann_G(\fb R[x,f]) \subseteq \Delta^S(G)$. Let $g \in
\Delta^S(G)$, so that there exist $s' \in S$ and $n_0 \in\nn$ such
that $s'x^ng = 0$ for all $n \geq n_0$. Thus $g \in
\ann_G(\bigoplus_{n\geq n_0}Rs'x^n) =: J$, a special annihilator
submodule of $G$. Let $\fb'$ be the $G$-special $R$-ideal that
corresponds to this special annihilator submodule (in the bijective
correspondence of \ref{np.2}(ii)). Since $\bigoplus_{n\geq
n_0}Rs'x^n \subseteq \grann_{R[x,f]}J = \fb' R[x,f]$, we have $s'
\in \fb'$, so that $\fb'\cap S \neq \emptyset$. Therefore $\fb'
\supseteq \fb$, and $g \in J = \ann_G(\fb' R[x,f]) \subseteq
\ann_G(\fb R[x,f]). $ Therefore $\Delta^S(G) \subseteq \ann_G(\fb
R[x,f])$. We conclude that $\Delta^S(G) = \ann_G(\fb R[x,f])$, and
that $(0:_R\Delta^S(G)) = (0:_R\ann_G(\fb R[x,f])) = \fb$. All the
claims in the statement of the proposition have now been proved.
\end{proof}

\begin{defi}
\label{st.8} An {\em $S$-test element for $R$} is an element $s \in
S$ such that, for every $R$-module $M$ and every $j \in \nn$, the
element $sx^j$ annihilates $1 \otimes m \in (\Phi(M))_0$ for every
$m \in 0^{*,S}_M$. The ideal of $R$ generated by all the $S$-test
elements for $R$ is called the {\em $S$-test ideal of $R$}, and
denoted by $\tau^S(R)$.
\end{defi}

One of the aims of this paper is to show that $S$-test elements for
$R$ exist when $R$ is $F$-pure and local.

This is a suitable point at which to remind the reader of some of
the classical concepts related to tight closure test elements.

\begin{rmd}
\label{inv.5p} Recall that a {\em test element for modules\/} for
$R$ is an element $c \in R^{\circ}$ such that, for every finitely
generated $R$-module $M$ and every $j \in \nn$, the element $cx^j$
annihilates $1 \otimes m \in (\Phi(M))_0$ for every $m \in 0^*_M$.
The phrase `for modules' is inserted because Hochster and Huneke
have also considered a concept of a {\em test element for ideals}
for $R$, which is defined to be an element $c \in R^{\circ}$ such
that, for every {\em cyclic\/} $R$-module $M$ and every $j \in \nn$,
the element $cx^j$ annihilates $1 \otimes m \in (\Phi(M))_0$ for
every $m \in 0^*_M$. When $R$ is reduced and excellent, the concepts
of test element for modules and test element for ideals for $R$
coincide: see \cite[Discussion (8.6) and Proposition
(8.15)]{HocHun90}.

Hochster and Huneke define the {\em test ideal $\tau(R)$ of $R$\/}
to be $\bigcap_M(0:_R0^*_M)$, where the intersection is taken over
all finitely generated $R$-modules $M$. In \cite[Proposition
(8.23)(b)]{HocHun90} they show that, if $R$ has a test element for
modules, then $\tau(R)$ is the ideal generated by the test elements
for modules, and $\tau(R)\cap R^{\circ}$ is the set of test elements
for modules.
\end{rmd}

\begin{rmd}
\label{inv.5r} Recall also that a {\em big test element\/} for $R$
is an element $c \in R^{\circ}$ such that, for {\em every\/}
$R$-module $M$ and every $j \in \nn$, the element $cx^j$ annihilates
$1 \otimes m \in (\Phi(M))_0$ for every $m \in 0^*_M$. We let
$\widetilde{\tau}(R)$ denote the ideal generated by all big test
elements for $R$, and call this the {\em big test ideal of $R$}.

Note that an $R^{\circ}$-test element for $R$, in the sense of
Definition \ref{st.8}, is just a big test element for $R$.
\end{rmd}

Recall that an {\em injective cogenerator of $R$\/} is an injective
$R$-module $E$ with the property that, for every $R$-module $M$ and
every non-zero element $m \in M$, there exists an $R$-homomorphism
$\phi : M \lra E$ such that $\phi(m) \neq 0$. As $R$ is Noetherian,
$\bigoplus_{\fm \in \Max(R)} E_R(R/\fm)$, where $\Max(R)$ denotes
the set of maximal ideals of $R$, is one injective cogenerator of
$R$.

\begin{rmds}
\label{inv.5q} Let $E := \bigoplus_{\fm \in \Max(R)} E_R(R/\fm)$, an
injective cogenerator of $R$.

\begin{enumerate}
\item Recall from Hochster--Huneke \cite[Definition (8.19)]{HocHun90} that
the {\em finitistic tight closure of\/ $0$ in $E$,\/} denoted by
$0^{*{\text f}{\text g}}_E$, is defined to be $\bigcup_{M}0^*_M$,
where the union is taken over all finitely generated $R$-submodules
$M$ of $E$. It was shown in \cite[Proposition (8.23)(d)]{HocHun90}
that $\tau (R) = (0:_R0^{*{\text f}{\text g}}_E)$.

\item It is conjectured that $(0:_R0^{*{\text f}{\text g}}_E) =
(0:_R0^*_E)$. This conjecture is known to be true
\begin{enumerate}
\item if $R$ is an excellent Gorenstein local ring (see K. E. Smith
\cite[p.\ 48]{Smith94});
\item if $R$ is the localization of a finitely generated $\nn$-graded
algebra over an $F$-finite field $K$ (of characteristic $p$ and
having $K$ as its component of degree $0$) at its unique homogeneous
maximal ideal (see G. Lyubeznik and K. E. Smith \cite[Corollary
3.4]{LyuSmi99});
\item if $R$ is a Cohen--Macaulay local ring which is Gorenstein on
its punctured spectrum (see Lyubeznik--Smith \cite[Theorem
8.8]{LyuSmi01}); or
\item if $(R, \fm)$ is local and an isolated singularity (see Lyubeznik--Smith
\cite[Theorem 8.12]{LyuSmi01}).
\end{enumerate}

\item It was shown in \cite[Theorem 3.3]{btctefsnrer}
that if $R$ has a big test element, then the big test ideal
$\widetilde{\tau}(R)$ of $R$ is equal to
$(0:_R\Delta^{R^{\circ}}(\Phi(E)))$, and the set of big test
elements for $R$ is $(0:_R\Delta^{R^{\circ}}(\Phi(E)))\cap
R^{\circ}$.
\end{enumerate}
\end{rmds}

\section{\it Existence of $S$-test elements in $F$-pure local rings}
\label{inv}

Theorem \ref{in.3} was proved by means of an `embedding theorem'
\cite[Theorem 4.10]{Fpurhastest}. Similar embedding theorems were
established in \cite[Theorem 3.5]{gatcti} and \cite[Theorem
3.2]{btctefsnrer}. In this paper, the ideas underlying those
theorems are going to be pursued further, and so we begin with three
remarks and a lemma that can be viewed as addenda to \cite[\S
2]{btctefsnrer}.

The notation in this section is as described in \ref{st.0}.

\begin{rmk}
\label{inv.1} Let $\widetilde{H}$ denote the graded companion of $H$
described in \cite[Example 2.1]{btctefsnrer}. It is easy to check
that $\Delta^S(\widetilde{H}) = \widetilde{\Delta^S(H)}$, so that
$(0:_R\Delta^S(\widetilde{H})) = (0:_R\Delta^S(H))$.
\end{rmk}

\begin{rmk}
\label{inv.2} Let $(H^{(\lambda)})_{\lambda\in\Lambda}$ be a
non-empty family of $\Z$-graded left $R[x,f]$-modules, and consider
the graded product $\prod^{\prime}_{\lambda\in\Lambda}H^{(\lambda)}$
of the $H^{(\lambda)}$, described in \cite[Lemma 2.1]{gatcti} and
\cite[Example 2.2]{btctefsnrer}. It is routine to check that, if
there is an ideal $\fd_0$ of $R$ such that
$(0:_R\Delta^S(H^{(\lambda)})) = \fd_0$ for all $\lambda\in\Lambda$,
then $\left( 0:_R
\Delta^S\left(\prod^{\prime}_{\lambda\in\Lambda}H^{(\lambda)}\right)\right)
= \fd_0$.
\end{rmk}

\begin{rmk}
\label{inv.3} Assume that the left $R[x,f]$-module $H =
\bigoplus_{n\in\Z}H_n$ is $\Z$-graded; let $t \in \Z$. Denote by
$H(t)$ the result of application of the $t$th shift functor to $H$;
this is described in \cite[Example 2.3]{btctefsnrer}. It is clear
that $\Delta^S(H(t)) = \Delta^S(H)(t)$, so that
$(0:_R\Delta^S(H(t))) = (0:_R\Delta^S(H))$.
\end{rmk}

\begin{lem}
\label{inv.4} Let $b, h \in \N$ and let $W := \bigoplus_{n \geq b}
W_n$ be a graded left $R[x,f]$-module. Let $(g_i)_{i\in I}$ be a
family of arbitrary elements of $W_b$. Consider the $h$-place
extension $\ext(W;(g_i)_{i\in I};h)$ of $W$ by $(g_i)_{i\in I}$,
defined in \cite[Definition 4.5]{Fpurhastest} and
\cite[\S2]{btctefsnrer}. Then $ (0:_R\Delta^S(\ext(W;(g_i)_{i\in
I};h))) = (0:_R\Delta^S(W)). $
\end{lem}

\begin{proof} This can be proved by making obvious modifications to
the proof, presented in \cite[Proposition 2.4(i)]{btctefsnrer}, that
$(0:_R\Delta^{R^{\circ}}(\ext(W;(g_i)_{i\in I};h))) =
(0:_R\Delta^{R^{\circ}}(W))$.
\end{proof}

The above remarks and lemma are helpful for use in the proof of some
of the claims in the following Embedding Theorem.

\begin{embthm}
\label{inv.5} {\rm (See \cite[Theorem 3.2]{btctefsnrer}.)} Let $E$
be an injective cogenerator of $R$. Assume that there exists an
$\nn$-graded left $R[x,f]$-module $H = \bigoplus_{n\in\nn}H_n$ such
that $H_0$ is $R$-isomorphic to $E$.

Let $M$ be an $R$-module. Then there is a family
$\left(L^{(n)}\right)_{n \in \nn}$ of\/ $\nn$-graded left
$R[x,f]$-modules, where $L^{(n)}$ is an $n$-place extension of the
$-n$th shift of a graded product of copies of $H$ (for each $n \in
\nn$), for which there exists a homogeneous $R[x,f]$-monomorphism
\[
\nu : \Phi(M) = \bigoplus_{i\in \nn}(Rx^i\otimes_RM) \lra
\prod_{n\in\nn}{\textstyle ^{^{^{\!\!\Large \prime}}}} L^{(n)} =: K.
\]
Consequently, $(0:_R\Delta^S(H)) = (0:_R\Delta^S(K)) \subseteq
(0:_R\Delta^S(\Phi(M)))$.

Furthermore, if $H$ is $x$-torsion-free, then so too is $\Phi(M)$,
and ${\mathcal I}(\Phi(M)) \subseteq {\mathcal I}(H)$ and
$(0:_R\Delta^S(\Phi(M))) \in {\mathcal I}(H)$.
\end{embthm}

\begin{proof} The existence of $K$ and $\nu$ with the stated properties were proved in
\cite[Theorem 3.2]{btctefsnrer}.

The existence of the $R[x,f]$-monomorphism $\nu$ shows that
$$
(0:_R\Delta^S(K)) \subseteq (0:_R\Delta^S(\Phi(M))),
$$
while  Remarks \ref{inv.2} and \ref{inv.3} and Lemma \ref{inv.4}
show that $(0:_R\Delta^S(K)) = (0:_R\Delta^S(H))$.

Now suppose $H$ is $x$-torsion-free. It follows from \cite[Lemmas
2.3 and 2.8]{gatcti} that $K$ is $x$-torsion-free and that
${\mathcal I}(K) = {\mathcal I}(H)$. The existence of the
$R[x,f]$-monomorphism $\nu$ shows that $\Phi(M)$ is $x$-torsion-free
and $R[x,f]$-isomorphic to an $R[x,f]$-submodule of $K$; therefore
${\mathcal I}(\Phi(M)) \subseteq {\mathcal I}(K) = {\mathcal I}(H)$.
Finally, $(0:_R\Delta^S(\Phi(M))) \in {\mathcal I}(\Phi(M))$.
\end{proof}

\begin{thm}
\label{fp.4} Suppose that the local ring $(R,\fm)$ is $F$-pure. Then
$R$ has an $S$-test element.

In more detail, set $E := E_R(R/\fm)$. Recall that $\fb^{S,\Phi(E)}$
denotes the intersection of all the minimal members of the set
$\left\{ \fp \in \Spec(R) \cap \mathcal{I}(\Phi(E)) : \fp \cap S
\neq \emptyset\right\} $ (see {\rm \ref{st.1}}). Then $S \cap
\fb^{S,\Phi(E)}$ is (non-empty and) equal to the set of $S$-test
elements for $R$. Furthermore, $\fb^{S,\Phi(E)} =
(0:_R\Delta^S(\Phi(E)))$.
\end{thm}

\begin{proof} By Theorem \ref{in.3}, the set
$\mathcal{I}(\Phi(E))$ of $\Phi(E)$-special $R$-ideals is finite.
Consequently, $\mathcal{M}^S(\Phi(E))$ is finite, and Proposition
\ref{st.5} shows that $(0:_R\Delta^S(\Phi(E))) \cap S \neq
\emptyset$.

We use the Embedding Theorem \ref{inv.5} with $\Phi(E)$ playing the
r\^ole of $H$. We conclude that, for every $R$-module $M$, we have
$$
(0:_R\Delta^S(\Phi(E))) \subseteq (0:_R\Delta^S(\Phi(M))).
$$
Thus $(0:_R\Delta^S(\Phi(E)))$ is precisely the set of elements of
$R$ that annihilate $\Delta^S(\Phi(M))$ for every $R$-module $M$.
Since $\Delta^S(\Phi(M))$ is an $R[x,f]$-submodule of $\Phi(M)$ (for
each $R$-module $M$), it follows that $(0:_R\Delta^S(\Phi(E)))\cap
S$ (which is non-empty) is the set of $S$-test elements for $R$.

It also follows from Proposition \ref{st.5} that
$$\Delta^S(\Phi(E)) =
\ann_{\Phi(E)}(\fb^{S,\Phi(E)} R[x,f])\quad \mbox{and}\quad
(0:_R\Delta^S(\Phi(E))) = \fb^{S,\Phi(E)}.$$
\end{proof}

In our applications of these ideas, we shall frequently take $S$ to
be the complement in $R$ of the union of finitely many prime ideals.
In that case, a little more can be said.

\begin{lem}
\label{fp.5} Let $A$ be a commutative ring and let $\fp_1, \ldots,
\fp_n \in \Spec (A)$. Set $T := A \setminus \bigcup_{i=1}^n\fp_i$,
and let $\fa$ be an ideal of $A$ such that $\fa \cap T \neq
\emptyset$. Then $\fa$ can be generated by elements of $\fa \cap T$.
\end{lem}

\begin{proof} Let $\fa'$ be the ideal of $A$ generated by
$\fa \cap T$. Then $$\fa \subseteq (\fa \cap T) \cup (\fa \cap (A
\setminus T)) \subseteq \fa' \cup \fp_1 \cup \cdots \cup \fp_n.$$
Since $\fa \cap T \neq \emptyset$, we have, for every $i \in \{1,
\ldots,n\}$, that $\fa \not\subseteq \fp_i$. Therefore, by the Prime
Avoidance Theorem (in the form given in \cite[Theorem 81]{IK}), we
must have $\fa \subseteq \fa'$.
\end{proof}

\begin{cor}
\label{fp.6} Suppose that the local ring $(R,\fm)$ is $F$-pure, and
$S = R \setminus \bigcup_{i=1}^n\fp_i$, where $\fp_1,
\ldots,\fp_n\in \Spec(R)$. It was shown in Theorem {\rm \ref{fp.4}}
that $R$ has an $S$-test element. Set $E := E_R(R/\fm)$. The
$S$-test ideal of $R$, that is, the ideal of $R$ generated by all
$S$-test elements for $R$, is $\fb^{S,\Phi(E)}$, the smallest member
of $\mathcal{I}(\Phi(E))$ that meets $S$. In symbols, $\tau^S(R) =
\fb^{S,\Phi(E)}$.

In particular, the big test ideal $\widetilde{\tau}(R)$ is the
smallest member of $\mathcal{I}(\Phi(E))$ of positive height. (We
interpret the height of the improper ideal $R$ as $\infty$.)
\end{cor}

\begin{proof} We saw in Theorem \ref{fp.4} that the set $S \cap
\fb^{S,\Phi(E)}$ is non-empty and equal to the set of $S$-test
elements for $R$. By Lemma \ref{fp.5}, the ideal $\fb^{S,\Phi(E)}$
can be generated by elements of $S \cap \fb^{S,\Phi(E)}$.

For the final claim, take $S = R^{\circ}$ and note that a proper
ideal of $R$ has positive height if and only if it meets
$R^{\circ}$.
\end{proof}

We shall actually use variations of the Embedding Theorem
\ref{inv.5} in Proposition \ref{inv.6} below.

\begin{defs}
\label{inv.6d} Suppose that $(R,\fm)$ is local and $F$-pure; set $E
:= E_R(R/\fm)$. Let $M$ be an $R$-module.
\begin{enumerate}
\item We define the {\em finitistic $S$-tight closure $0_M^{*{\rm fg},S}$
of $0$ in $M$\/} to be $\bigcup_N 0^{*,S}_N$, where the union is
taken over all finitely generated submodules $N$ of $M$.
\item We define the {\em finitistic $S$-test ideal $\tau^{{\rm fg},S} (R)$ of $R$\/} to be
$\bigcap_L(0:_R0^{*,S}_L)$, where the intersection is taken over all
finitely generated $R$-modules $L$.
\end{enumerate}
\end{defs}

\begin{prop}
\label{inv.6} Suppose that $(R,\fm)$ is local and $F$-pure; set $E
:= E_R(R/\fm)$.
\begin{enumerate}
\item For every $R$-module $M$, we have ${\mathcal I}(\Phi(M))
\subseteq {\mathcal I}(\Phi(E))$ and $(0:_R\Delta^S(\Phi(E)))
\subseteq (0:_R\Delta^S(\Phi(M))) \in {\mathcal I}(\Phi(E))$.

\item The ideal $(0:_R0^{*{\rm
f}{\rm g},S}_E)$ annihilates $(0:_R0^{*,S}_L)$ for every finitely
generated $R$-module $L$.

\item We have $\tau^{{\rm fg},S} (R) = (0:_R0^{*{\rm
f}{\rm g},S}_E)$, and this ideal belongs to ${\mathcal I}(\Phi(E))$.

\item For every $R$-module $M$, we have $(0:_R0^{*,S}_E)
\subseteq (0:_R0^{*,S}_M)$.

\item We have $\fb^{S,\Phi(E)} = (0:_R0^{*,S}_E)$. Consequently, when
$S$ is the complement in $R$ of the union of finitely many prime
ideals, then the $S$-test ideal $\tau^S(R)$ of $R$ is equal to
$(0:_R0^{*,S}_E)$.

\end{enumerate}
\end{prop}

\begin{proof} (i) This follows from the Embedding
Theorem \ref{inv.5}, with $H$ taken as $\Phi(E)$.

(ii) By Krull's Intersection Theorem, $\bigcap_{n\in\N}\fm^nL = 0$.
We can therefore express the zero submodule of $L$ as the
intersection of a countable family $(Q_{i})_{i\in \N}$ of
irreducible submodules of finite colength. (A submodule of $L$ is
said to be {\em irreducible\/} if it is proper and cannot be
expressed as the intersection of two strictly larger submodules.)
Note that $E_R(L/Q_i) = E$ for all $i \in \N$.

The $R$-monomorphism $\Lambda_0 : L \lra \prod_{i\in\N}L/Q_i$ for
which $\lambda_0(g) = (g + Q_i)_{i\in \N}$ for all $g \in L$ can be
extended to a homogeneous $R[x,f]$-homomorphism
$$
\lambda : \Phi(L) = \bigoplus_{n\in\nn} Rx^n\otimes_RL \lra
\prod_{i\in\N}{\textstyle ^{^{^{\!\!\Large \prime}}}} \Phi(L/Q_i) =
\prod_{i\in\N}{\textstyle ^{^{^{\!\!\Large \prime}}}}
\left(\bigoplus_{n\in\nn} Rx^n\otimes_R(L/Q_i)\right)
$$
whose restriction to the $n$th component of $\Phi(L)$, for $n \in
\nn$, satisfies $\lambda(rx^n\otimes g) = \left(rx^n\otimes (g +
Q_i)\right)_{i\in\N}$ for all $r \in R$ and $g \in L$. It is clear
that the $R$-monomorphism $\lambda_0$ (that is, the component of
degree $0$ of $\lambda$) maps $0^{*,S}_L$ into $\prod_{i\in
\N}0^{*,S}_{L/Q_i}$. But $L/Q_i$ is $R$-isomorphic to a finitely
generated submodule of $E$, and so $0^{*,S}_{L/Q_i}$ is annihilated
by $(0:_R0^{*{\rm f}{\rm g},S}_E)$ (for all $i \in \N$). It follows
that $0^{*,S}_L$ is annihilated by $(0:_R0^{*{\rm f}{\rm g},S}_E)$.

(iii) By part (ii), we have $(0:_R0^{*{\rm f}{\rm g},S}_E) \subseteq
\bigcap_L(0:_R0^{*,S}_L)$, where the intersection is taken over all
finitely generated $R$-modules $L$. Thus $(0:_R0^{*{\rm f}{\rm
g},S}_E) \subseteq \tau^{{\rm fg},S} (R)$. On the other hand, by
definition, $\tau^{{\rm fg},S} (R)$ annihilates $\bigcup_N
0^{*,S}_N$, where the union is taken over all finitely generated
submodules $N$ of $E$; therefore $\tau^{{\rm fg},S} (R) \subseteq
(0:_R0^{*{\rm f}{\rm g},S}_E)$.

For each finitely generated $R$-module $L$,
$$\Delta^{S}(\Phi(L)) = 0^{*,S}_L \oplus 0^{*,S}_{Rx \otimes_RL}
\oplus \cdots \oplus 0^{*,S}_{Rx^n \otimes_RL}\oplus \cdots, $$ by
Example \ref{st.3}(ii), so that $ (0:_R\Delta^{S}(\Phi(L))) =
\bigcap_{n\in\nn} (0:_R0^{*,S}_{Rx^n \otimes_RL}). $ Note that $Rx^n
\otimes_RL$ is a finitely generated $R$-module, for each $n \in
\nn$. It therefore follows that
$$
{\textstyle \bigcap_L(0:_R\Delta^{S}(\Phi(L))) =
\bigcap_L(0:_R0^{*,S}_L)},
$$
where in both cases the intersection is taken over all finitely
generated $R$-modules $L$. Therefore $\tau^{{\rm fg},S} (R)$ is
equal to the above ideal $\bigcap_L(0:_R\Delta^{S}(\Phi(L)))$, and
the latter is in ${\mathcal I}(\Phi(E))$ because each
$(0:_R\Delta^{S}(\Phi(L)))$ is (by part (i)) and ${\mathcal
I}(\Phi(E))$ is closed under taking arbitrary intersections (by
\cite[Corollary 1.12]{ga}).

(iv) By \cite[Lemma 3.1]{btctefsnrer}, there is a family of graded
left $R[x,f]$-modules
$\left(H^{(\lambda)}\right)_{\lambda\in\Lambda}$, with each
$H^{(\lambda)}$ equal to $\Phi(E)$, and a homogeneous
$R[x,f]$-homomorphism
$$
\mu : \Phi(M) \lra \prod_{\lambda\in\Lambda}{\textstyle
^{^{^{\!\!\Large \prime}}}}H^{(\lambda)}
$$
such that its component $\mu_0$ of degree $0$ is a monomorphism.
Since $\mu_n(sx^n\otimes m) = sx^n \mu_0(m)$ for all $m \in M$,
$s\in S$ and $n \in \nn$, it follows that the $R$-monomorphism
$\mu_0$ maps $0^{*,S}_M$ into a direct product of copies of
$0^{*,S}_E$. Therefore $(0:_R0^{*,S}_E)$ annihilates $0^{*,S}_M$.
Note that this is true for each $R$-module $M$.

(v) It now follows from part (iv) and the fact (see Example
\ref{st.3}(ii)) that
$$\Delta^{S}(\Phi(E)) = 0^{*,S}_E \oplus 0^{*,S}_{Rx
\otimes_RE} \oplus \cdots \oplus 0^{*,S}_{Rx^n \otimes_RE}\oplus
\cdots
$$ that $(0:_R0^{*,S}_E) = (0:_R\Delta^{S}(\Phi(E)))$. But $(0:_R\Delta^{S}(\Phi(E))) =
\fb^{S,\Phi(E)}$, by Theorem \ref{fp.4}.

When $S$ is the complement in $R$ of the union of finitely many
prime ideals, it follows from Corollary \ref{fp.6} that $\tau^S(R)=
\fb^{S,\Phi(E)}$.
\end{proof}

\begin{rmks}
\label{inv.7} Suppose that $(R,\fm)$ is local and $F$-pure; set $E
:= E_R(R/\fm)$.
\begin{enumerate}
\item In the special case in which $S$ is taken to be $R^{\circ}$,
Proposition \ref{inv.6}(v) reduces to the (probably well-known)
result that the big test ideal $\widetilde{\tau}(R) =
\tau^{R^{\circ}}(R)$ of $R$ is equal to $(0:_R0^{*}_E)$.
\item In the special case in which $S$ is taken to be $R^{\circ}$,
the first part of Proposition \ref{inv.6}(iii) reduces to (a special
case of) a result of Hochster--Huneke \cite[Proposition
(8.23)(d)]{HocHun90} about the test ideal $\tau(R)$:
$$
\tau(R) = \tau^{{\rm fg},R^{\circ}} (R) = (0:_R0^{*{\rm f}{\rm
g},R^{\circ}}_E) = (0:_R0^{*{\rm f}{\rm g}}_E).
$$
\end{enumerate}
\end{rmks}

We have seen that, over an $F$-pure local ring $(R,\fm)$, the set
$\mathcal{I}(\Phi(E))$ (where $E := E_R(R/\fm)$) of radical ideals
includes the test ideal $\tau(R)$, the big test ideal
$\widetilde{\tau}(R) = \fb^{R^{\circ},\Phi(E)} =
\tau^{R^{\circ}}(R)$ of $R$ and the $S$-test ideal, for each
multiplicatively closed subset $S$ of $R$ which is the complement in
$R$ of the union of finitely many prime ideals. It is natural to ask
whether every member of the finite set $\mathcal{I}(\Phi(E))$ occurs
as the $S'$-test ideal for some multiplicatively closed subset $S'$
of $R$. In Theorem \ref{fp.6a} below, we shall answer this question
in the affirmative.

\begin{thm}
\label{fp.6a} Suppose that the local ring $(R,\fm)$ is $F$-pure, and
set $E := E_R(R/\fm)$. Let $\fa \in \mathcal{I}(\Phi(E))$. Then
there exists a multiplicatively closed subset $S$ of $R$ such that
$\fa$ is the $S$-test ideal of $R$. Moreover, $S$ can be taken to be
the complement in $R$ of the union of finitely many prime ideals.
\end{thm}

\begin{proof} If $\fa = R$, then we can take $S = \{1\}$ or $S = R \setminus \fm$. We therefore
assume henceforth in this proof that $\fa $ is proper.

Let $\fp_1, \ldots, \fp_t$ be the (distinct) associated prime ideals
of $\fa$; recall from \cite[Theorem 3.6 and Corollary 3.7]{ga} that
they all belong to $\mathcal{I}(\Phi(E))$. Let $\mathcal{T}$ be the
set of all prime members of $\mathcal{I}(\Phi(E))$ which neither
contain, nor are contained in, any of $\fp_1, \ldots, \fp_t$. Let
$\fq_1, \ldots, \fq_u$ be the maximal members of the set of prime
ideals in $\mathcal{I}(\Phi(E))$ that are properly contained in one
of $\fp_1, \ldots, \fp_t$, and let $\mathcal{U} := \{ \fq_1, \ldots,
\fq_u\}$. (Observe that ${\mathcal T}$ and/or ${\mathcal U}$ could
be empty; for example, both are empty if $\fa = 0$.)

Set $S := R \setminus \bigcup_{\fq \in \mathcal{T}\cup
\mathcal{U}}\fq$. Our aim is to show that $\fa$ is the $S$-test
ideal $\tau^S(R)$ of $R$. It follows from Corollary \ref{fp.6} that
$\tau^S(R) = \fb^{S,\Phi(E)}$, the intersection of the minimal
members of the set of prime ideals in $\mathcal{I}(\Phi(E))$ that
meet $S$.

Let $\fp \in \mathcal{I}(\Phi(E))\cap \Spec(R)$. Then, since
$\mathcal{I}(\Phi(E))$ and therefore $\mathcal{T}$ and $\mathcal{U}$
are finite, $\fp \cap S = \emptyset$ if and only if $\fp$ is
contained in some $\fq \in \mathcal{T}\cup \mathcal{U}$.

Let $i \in \{1,\ldots,t\}$. Then $\fp_i$ meets $S$, or else $\fp_i
\subseteq \fq_j$ for some $j \in \{1, \ldots, u\}$, and as $\fq_j$
is properly contained in one of $\fp_1, \ldots, \fp_t$, this would
lead to a contradiction to the minimality of the primary
decomposition $\fa = \fp_1\cap \cdots \cap \fp_t$. Furthermore,
$\fp_i$ must be a minimal member of the set $ \mathcal{J} := \left\{
\fp \in \mathcal{I}(\Phi(E))\cap\Spec(R) : \fp \cap S \neq \emptyset
\right\},$ for otherwise there would exist $\fp \in \mathcal{J}$
with $\fp \subset \fp_i$, so that $\fp$ would be contained in one of
$\fq_1, \ldots, \fq_u$ and therefore disjoint from $S$. This shows
that $\fp_1, \ldots, \fp_t$ are all associated primes of
$\fb^{S,\Phi(E)}$. To complete the proof, it is enough for us to
show that there is no other associated prime of this ideal.

So suppose that $\fp \in \ass \fb^{S,\Phi(E)} \setminus \{\fp_1,
\ldots,\fp_t\}$ and seek a contradiction. Then $\fp$ must contain,
or be contained in, $\fp_i$ for some $i \in \{1, \ldots, t\}$ (or
else it would be in $\mathcal{T}$ and disjoint from $S$); if $\fp
\subset \fp_i$, then $\fp$ would be contained in $\fq_j$ for some $j
\in \{1, \ldots, t\}$ and so would be disjoint from $S$; if $\fp
\supset \fp_i$, then $\fp$ could not be a primary component of the
radical ideal $\fb^{S,\Phi(E)}$. Thus each possibility leads to a
contradiction. Therefore $\ass \fb^{S,\Phi(E)} = \{\fp_1, \ldots,
\fp_t\}$ and $\fb^{S,\Phi(E)} = \fa$.
\end{proof}

\section{\it The complete case}
\label{fp}

In this section we shall concentrate on the case where $(R,\fm)$ is
local, $F$-pure and complete.

\begin{thm}
\label{fp.11} Suppose $(R,\fm)$ is local, $F$-pure and complete. Set
$E := E_R(R/\fm)$. Let $\fc \in {\mathcal I}(\Phi(E))$ with $\fc
\neq R$. In the light of Theorem\/ {\rm \ref{fp.6a}}, let $\fp_1,
\ldots,\fp_w$ be prime ideals of $R$ for which the multiplicatively
closed subset $S = R \setminus \bigcup_{i=1}^w\fp_i$ of $R$
satisfies $\fc = \tau^{S}(R)$. Set $ J := \Delta^{S}(\Phi(E)),$ a
graded left $R[x,f]$-module.
\begin{enumerate}
\item We have
$
J = 0^{*,S}_{E} \oplus 0^{*,S}_{Rx\otimes_RE} \oplus \cdots \oplus
0^{*,S}_{Rx^n\otimes_RE} \oplus \cdots.
$
\item When we regard $J$ as a graded left
$(R/\fc)[x,f]$-module in the natural way, it is $x$-torsion-free and
has ${\mathcal I}_{R/\fc}(J) = \left\{\fg/\fc : \fg \in {\mathcal
I}(\Phi(E)) : \fg \supseteq \fc \right\}.$
\item The $0$th component $J_0$ of $J$ is $(0:_{E}\fc)$; as
$R/\fc$-module, this is isomorphic to
$E_{R/\fc}((R/\fc)/(\fm/\fc))$.
\item The ring $R/\fc$ is $F$-pure.
\item We have ${\mathcal I}(\Phi_{R/\fc}(J_0)) \subseteq {\mathcal
I}_{R/\fc}(J)$, so that
$$
\left\{ \fd : \fd \mbox{~is an ideal of $R$ with~} \fd \supseteq \fc
\mbox{~and~} \fd/\fc \in {\mathcal I}(\Phi_{R/\fc}(J_0))\right\}
\subseteq {\mathcal I}(\Phi_R(E)).
$$
\end{enumerate}
\end{thm}

\begin{proof} Set $\overline{R} := R/\fc$.

(i) It follows from Example \ref{st.3}(ii) that
$$
\Delta^{S}(\Phi(E)) = 0^{*,S}_{E} \oplus
0^{*,S}_{Rx\otimes_RE}\oplus \cdots \oplus 0^{*,S}_{Rx^n\otimes_RE}
\oplus \cdots.
$$

(ii) Since $\grann_{R[x,f]}\Delta^{S}(\Phi(E)) = \fc R[x,f]$, we see
that $\fc$ annihilates $\Delta^{S}(\Phi(E))$, and so the latter
inherits a structure as an $x$-torsion-free left
$\overline{R}[x,f]$-module. As the $R[x,f]$-submodules of $J$ are
exactly the $R[x,f]$-submodules of $\Phi(E)$ contained in $J$, the
claim about ${\mathcal I}_{\overline{R}}(J)$ is clear.

(iii) Note that $\fc = \fb^{S,\Phi(E)} = \tau^S(R)$, and that, by
Proposition \ref{inv.6}(v), this is the $R$-annihilator of
$0^{*,S}_{E}$. Since $R$ is complete, we can conclude that
$0^{*,S}_{E} = (0:_E\fc)$, by Matlis duality (see, for example,
\cite[p.\ 154]{SV}).

(iv),(v) Let $N$ be an $\overline{R}$-module. Use the Embedding
Theorem \ref{inv.5} over the ring $\overline{R}$ (with $J$ playing
the r\^ole of $H$) to deduce that $\Phi_{\overline{R}}(N)$ is
$x$-torsion-free and ${\mathcal I}(\Phi_{\overline{R}}(N)) \subseteq
{\mathcal I}_{\overline{R}}(J)$. These are true for each
$\overline{R}$-module $N$, and, in particular, for $J_0$. It follows
that $\overline{R}$ is $F$-pure.

The final claim follows from the description of ${\mathcal
I}_{\overline{R}}(J)$ given in part (ii).
\end{proof}

\begin{cor}
\label{fp.12} Suppose that $(R,\fm)$ is local, $F$-pure and
complete. Denote $E_R(R/\fm)$ by $E$. Let $\fc \in
\mathcal{I}(\Phi(E))$ with $\fc \neq R$. Denote $R/\fc$ by
$\overline{R}$, and note that $\overline{R}$ is $F$-pure, by
Theorem\/ {\rm \ref{fp.11}(iv)}. Let $T$ be a multiplicatively
closed subset of $\overline{R}$ which is the complement in
$\overline{R}$ of the union of finitely many prime ideals.

\begin{enumerate}
\item If\/ $\fh$ denotes the unique ideal of $R$ that contains $\fc$ and
is such that $\fh/\fc = \tau^{{\rm fg},T}(\overline{R})$, the
finitistic $T$-test ideal of $\overline{R}$, then $\fh \in
\mathcal{I}(\Phi(E))$.
\item In particular, if\/ $\fh'$ denotes the unique ideal of $R$ that contains $\fc$ and
is such that $\fh'/\fc = \tau(\overline{R})$, the test ideal of
$\overline{R}$, then $\fh' \in \mathcal{I}(\Phi(E))$.
\item If\/ $\fg$ denotes the unique ideal of $R$ that contains $\fc$ and
is such that $\fg/\fc = \tau^T(\overline{R})$, the $T$-test ideal of
$\overline{R}$, then $\fg \in \mathcal{I}(\Phi(E))$.
\item In particular, if\/ $\fg'$ denotes the unique ideal of $R$ that contains $\fc$ and
is such that $\fg'/\fc = \widetilde{\tau}(\overline{R})$, the big
test ideal of $\overline{R}$, then $\fg' \in \mathcal{I}(\Phi(E))$.
\end{enumerate}
\end{cor}

\begin{proof} Use the notation of Theorem \ref{fp.11}. Note that, as $\overline{R}$-module, $J_0$ is
the injective envelope of the simple $\overline{R}$-module.

(i) By Proposition \ref{inv.6}(iii), we have $\tau^{{\rm
fg},T}(\overline{R}) \in {\mathcal I}(\Phi_{\overline{R}}(J_0))$.
The result therefore follows from Theorem \ref{fp.11}(v).

(ii) This is a special case of part (i): take $T =
\overline{R}^{\circ}$.

(iii) By Proposition \ref{inv.6}(v), we have $\tau^{T}(\overline{R})
\in {\mathcal I}(\Phi_{\overline{R}}(J_0))$. The result therefore
follows from Theorem \ref{fp.11}(v).

(iv) This is a special case of part (iii): take $T =
\overline{R}^{\circ}$.
\end{proof}

The remainder of the paper is devoted to the provision of some
examples of the above ideas.

\begin{ex}
\label{fp.14a} Let $K$ be an algebraically closed field of
characteristic $p$, and assume that $p \geq 5$ and that $p \equiv 1
\; (\mbox{mod}~ 3)$. Let $R' = K[[X,Y,Z]]$, where $X,Y,Z$ are
independent indeterminates, and $\fa = (X^3 + Y^3 + Z^3) \in \Spec
(R')$. By Huneke \cite[Examples 4.7 and 4.8]{Hunek98}, $R := R'/\fa$
is $F$-pure, and the test ideal $\tau(R) = \fm$. Because $R$ is
Gorenstein and excellent, the test ideal $\tau(R)$ is equal to the
big test ideal $\widetilde{\tau}(R)$, by \ref{inv.5q}(ii)(a). This
means that $\fm$ must be the smallest ideal in
$\mathcal{I}(\Phi(E))$ of positive height, so that
$
\mathcal{I}(\Phi(E))\cap \Spec(R) = \left\{ 0, \fm \right\}.
$
\end{ex}

\begin{rmd}
\label{fp.15} Suppose that $(R,\fm)$ is local and $F$-pure, and set
$E = E_R(R/\fm)$.

In the case where $R$ is an ($F$-pure) homomorphic image of an
$F$-finite regular local ring, Janet Cowden Vassilev showed in
\cite[\S3]{Cowde98} that there exists a strictly ascending chain $0
= \tau_0 \subset \tau_1 \subset \cdots \subset \tau_t \subset
\tau_{t+1} = R$ of radical ideals of $R$ such that, for each $i = 0,
\ldots, t$, the reduced local ring $R/\tau_i$ is $F$-pure and its
test ideal is exactly $\tau_{i+1}/\tau_i$. If $R$ is complete, all
of $\tau_0, \tau_1, \ldots, \tau_t$ and all their associated primes
belong to ${\mathcal I}(\Phi(E))$ (by Corollary \ref{fp.12}(ii) and
\cite[Theorem 3.6 and Corollary 3.7]{ga}).
\end{rmd}

\begin{lem}
\label{fp.16} Assume that $(R,\fm)$ is local, $F$-pure and complete.
Set $E = E_R(R/\fm)$.
\begin{enumerate}
\item There is a strictly ascending chain $0
= \tau_0 \subset \tau_1 \subset \cdots \subset \tau_t \subset
\tau_{t+1} = R$ of radical ideals of $R$ such that, for each $i = 0,
\ldots, t$, the reduced local ring $R/\tau_i$ is $F$-pure and its
test ideal is $\tau_{i+1}/\tau_i$. We call this the {\em test ideal
chain of $R$}. All of $\tau_0 = 0,\tau_1, \cdots ,\tau_t$, and all
their associated primes, belong to ${\mathcal I}(\Phi(E))$.
\item There is a strictly ascending chain $0
= \widetilde{\tau}_0 \subset \widetilde{\tau}_1 \subset \cdots
\subset \widetilde{\tau}_w \subset \widetilde{\tau}_{w+1} = R$ of
radical ideals in $\mathcal{I}(\Phi(E))$ such that, for each $i = 0,
\ldots, w$, the reduced local ring $R/\widetilde{\tau}_i$ is
$F$-pure and its big test ideal is
$\widetilde{\tau}_{i+1}/\widetilde{\tau}_i$. We call this the {\em
big test ideal chain of $R$}. All of $\widetilde{\tau}_0 =
0,\widetilde{\tau}_1, \cdots ,\widetilde{\tau}_w$, and all their
associated primes, belong to ${\mathcal I}(\Phi(E))$.
\end{enumerate}
\end{lem}

\begin{note} We have not assumed that $R$ is $F$-finite in part (i).
If the conjecture mentioned in \ref{inv.5q}(ii) turns
out to be true, then the big test ideal chain and the test ideal
chain of $R$ would coincide.
\end{note}

\begin{proof} (i) We know from Theorem \ref{fp.2}
that $E$ can be given a structure as an $x$-torsion-free left
$R[x,f]$-module (that extends its $R$-module structure). It
therefore follows from work of the present author in \cite[Corollary
3.8]{gatcti} that, in this complete case, the test ideal chain of
$R$ exists. By Corollary \ref{fp.12}(ii), all the terms in the test
ideal chain of $R$ belong to ${\mathcal I}(\Phi(E))$, and all the
associated prime ideals of these ideals also belong to ${\mathcal
I}(\Phi(E))$, by \cite[Theorem 3.6 and Corollary 3.7]{ga}.

(ii) Let $\fc \in {\mathcal I}(\Phi(E))$ with $\fc \neq R$. By
Theorem \ref{fp.11}(iv), the ring $R/\fc$ is $F$-pure. By Corollary
\ref{fp.12}(iv), if $\fg'$ denotes the unique ideal of $R$ that
contains $\fc$ and is such that $\fg'/\fc =
\widetilde{\tau}(R/\fc)$, then $\fg' \in \mathcal{I}(\Phi(E))$. One
can therefore construct $\widetilde{\tau}_1(R),
\widetilde{\tau}_2(R), \ldots$ successively until some
$\widetilde{\tau}_{t+1}(R) = R$, when the process stops. Use
Corollary \ref{fp.12}(iv) and \cite[Theorem 3.6 and Corollary
3.7]{ga} again to complete the proof.
\end{proof}

We can now use some of Vassilev's computations in \cite[\S
3]{Cowde98} to give some examples.

\begin{exs}
\label{fp.17} Let $K$ be an algebraically closed field of
characteristic $p$. In these examples, $X,Y,Z,W$ denote independent
indeterminates over $K$, and $x,y,z,w$ denote the natural images of
$X,Y,Z,W$ (respectively) in $R'/\fa = R$ for appropriate choices of
$R'$ and a proper ideal $\fa$ of $R'$.

\begin{enumerate}
\item As in Vassilev \cite[Example 3.12(1)]{Cowde98}, take $$R' =
K[[X,Y,Z]], \quad \fa = (XY,XZ,YZ)\quad \mbox{and}  \quad R =
R'/\fa.$$ For this $R$, the test ideal chain is $0 \subset \fm
\subset R$. Since $R$ is an isolated singularity, we have $\tau(R) =
\widetilde{\tau}(R)$, by \ref{inv.5q}(ii)(d). Therefore $\fm$ is the
smallest ideal of positive height in ${\mathcal I}(\Phi(E))$. Thus
in this case,
$$
{\mathcal I}(\Phi(E)) \cap \Spec (R) = \left\{ (x,y),(x,z),(y,z),
\fm \right\}.
$$
\item As in Vassilev \cite[Example 3.12(2)]{Cowde98}, take $$R' =
K[[X,Y,Z,W]], \quad \fa = (XYZ,XYW,XZW,YZW)\quad \mbox{and}  \quad R
= R'/\fa.$$ For this $R$, the test ideal chain is $0 \subset
(xy,xz,xw,yz,yw,zw) \subset \fm \subset R$. It therefore follows
from Lemma \ref{fp.16}(i) that
\begin{align*}
\left\{ \right. (x,y), (x,z), (x,w), (y,z),&  (y,w), (z,w), (x,y,z),
(x,y,w),(x,z,w),
 \\ & (y,z,w),\left.  \fm \right\} \subseteq
{\mathcal I}(\Phi(E)) \cap \Spec (R).
\end{align*}
\item As in Vassilev \cite[Example 3.12(3)]{Cowde98}, take $R' =
K[[X,Y,Z]]$, $\fa = (XY,YZ)$, and $R = R'/\fa$. For this $R$, the
test ideal chain is $0 \subset \fm \subset R$. Because $R$ is an
isolated singularity, $\tau(R) = \widetilde{\tau}(R)$, by
\ref{inv.5q}(ii)(d). This means that $\fm$ must be the smallest
ideal in $\mathcal{I}(\Phi(E))$ of positive height, so that $
{\mathcal I}(\Phi(E)) \cap \Spec (R) = \left\{ (x,z), (y), \fm
\right\}.$
\item As in Vassilev \cite[Example 3.12(4)]{Cowde98}, take $$R' =
K[[X,Y,Z,W]],\quad \fa = (XY,ZW)\quad \mbox{and} \quad R = R'/\fa.$$
For this $R$, the test ideal chain is $0 \subset (xy,xz,xw,yz,yw,zw)
\subset \fm \subset R,$ so that
\begin{align*}
\left\{ (x,z), (x,w), (y,z), (y,w), (x,y,z), (x,y,w), \right. &
\left. \! \! (x,z,w), (y,z,w), \fm \right\} \\ &   \subseteq
{\mathcal I}(\Phi(E)) \cap \Spec (R)
\end{align*}
by Lemma \ref{fp.16}(i).
\end{enumerate}
\end{exs}

\end{document}